\documentclass[12pt]{elsarticle}
\usepackage[intlimits]{amsmath}
\usepackage[english]{babel}
\usepackage[utf8]{inputenc}
\usepackage{amssymb,amsfonts}
\usepackage{mathrsfs}
\usepackage{amsthm}
\usepackage{longtable}
\usepackage{indentfirst}
\usepackage{tabularx} 
\usepackage{multicol}

\theoremstyle{definition}
\newtheorem{theorem}{Theorem}
\newtheorem*{theorem*}{Theorem}
\newtheorem{lemma}{Lemma}
\newtheorem*{proposition}{Proposition}

\newtheorem*{definition}{Definition}

\newtheorem{conjecture}{Conjecture}

\usepackage{hyperref}
\hypersetup{
    colorlinks=true,
    linkcolor=blue,
    filecolor=magenta,      
    urlcolor=cyan,
}

\makeatletter
\def\ps@pprintTitle{%
  \let\@oddhead\@empty
  \let\@evenhead\@empty
  \let\@oddfoot\@empty
  \let\@evenfoot\@oddfoot
}
\makeatother

\renewcommand{\leq}{\leqslant}

\renewcommand{\geq}{\geqslant}

\renewcommand{\phi}{\varphi}
\DeclareMathOperator{\Nu}{\mathcal{V}}

\DeclareMathOperator{\Q}{\mathbb{Q}}
\DeclareMathOperator{\R}{\mathbb{R}}

\DeclareMathOperator{\N}{\mathbb{N}}

\DeclareMathOperator{\x}{\textbf{x}}
\DeclareMathOperator{\y}{\textbf{y}}
\DeclareMathOperator{\e}{\textbf{e}}
\DeclareMathOperator{\zero}{\textbf{0}}

\DeclareMathOperator{\diag}{diag}

\DeclareMathOperator{\embed}{\stackrel{\Q}{\hookrightarrow}}
\DeclareMathOperator{\notembed}{\stackrel{\Q}{\not \hookrightarrow}}
\DeclareMathOperator{\eqQ}{\stackrel{\Q}{\sim}}
\DeclareMathOperator{\mcQ}{\mathcal{Q}}

\title{On distance graphs in rational spaces}

\author{A.\,A. Sokolov\footnote{Moscow Institute of Physics and Technology, Department of Discrete Mathematics,  \href{mailto:sokolov.aa@phystech.edu}{sokolov.aa@phystech.edu}
}}
    
\date{}

\begin{document}

\begin{abstract}
For any positive definite rational quadratic form $q$ of $n$ variables let $G(\mathbb{Q}^n, q)$ denote the graph with vertices $\Q^n$ and $x, y \in \Q^n$ connected iff $q(x - y) = 1$. This notion generalises standard Euclidean distance graphs.
In this article we study these graphs and show how to find the exact value of clique number of the $G(\mathbb{Q}^n, q)$.

We also prove rational analogue of the Beckman--Quarles theorem that any unit-preserving mapping of $\mathbb{Q}^n$ is an isometry.

\end{abstract}


\maketitle

\section{Introduction}\label{intro_chapter}



Let $\mcQ_n$ be a set of all rational positive definite quadratic forms on $n$ variables.

\begin{definition}
For every $n \in \N$ and $q \in \mcQ_n$ graph $G(\Q^n, q) = (V, E)$ is defined as follows:
\[
V = \Q^n, \quad E = \{(\x, \y) \quad | \quad q(\x - \y) = 1\}.
\]
\end{definition}

This definition generalises standard Euclidean distance graph.
Indeed, if $I_n(x_1, \ldots, x_n) = x_1^2 + \ldots + x_n^2$ is standard Euclidean quadratic form, then  $G(\Q^n, I_n)$ and $G(\Q^n, \frac{1}{d}I_n)$ are standard Euclidean distance graphs with distance 1 and $\sqrt{d}$ respectively.

In section \ref{defs_and_propert_chapter} we study basic properties of such graphs i.e. the non-emptiness and the connectedness. 

In section \ref{clique_chapter} we study clique numbers of such graphs, i.e. the maximal number of vertices connected to each other.
An algorithm on finding exact value of clique number of $G(\Q^n, q)$ any $n$ and $q$ is presented.

This algorithm was implemented for the case $q = \frac{1}{d}I_n$ using SageMath, respective computer program can be found at \cite{SokolovGit}.
More about values of $\omega(G(\Q^n, I_n))$ and $\omega(G(\Q^n, \frac{1}{d}I_n))$ reader can find in \cite{chilakamarri1988unit}, \cite{ElsholtzKlotz2005}, \cite{bau2021single}, \cite{Noble2018}, \cite{noble2021embedding}.

In section \ref{isom_chapter} we state the following question:

"Given two rational quadratic forms $q_1$ and $q_2$, are graphs $G(\Q^n, q_1)$ and $G(\Q^n, q_2)$ isomoprhic or not? If so, what do these isomorphisms look like?"

This question is closely related to the rational analogue of Beckman--Quarles Theorem which states that any automorphism of $G(\Q^n, I_n)$ is an isometry.
This analogue is proven for all $n$.

More about Beckman--Quarles theorem reader can find in \cite{beckman1953isometries}, \cite{zaks2001discrete}, \cite{zaks2003beckman}, \cite{zaksconnelly2003beckman}.





Through this article notions and theorems from the theory of rational quadratic forms will be used which can be found in \ref{appendix}.

\section{Definitions and main properties}\label{defs_and_propert_chapter}

Let $\mcQ_n$ be the set of all rational positive definite quadratic forms on $n$ variables. Let $\mcQ = \cup_{n = 1}^{\infty} \mcQ_n$.

\begin{definition}\label{main_def}
For every  $q \in \mcQ_n$ a $G(\Q^n, q) = (V, E)$ is defined as follows:
\[
V = \Q^n, \quad E = \{(\x, \y) \quad | \quad q(\x - \y) = 1\}.
\]
\end{definition}



\begin{definition}
    Forms $q_1$ and $q_2$ are called \textit{rationally equivalent} and denoted as  $q_1 \eqQ q_2$ iff $\dim q_1 = \dim q_2$ and  there exists linear transformation $T \in GL\left(n, \Q\right)$ such that  $q_1(\x) = q_2(T\x)$ and call such forms .
\end{definition}

\begin{lemma} \label{isom}
	If $q_1 \eqQ q_2$ then corresponding graphs $G(\Q^{\dim q_1}, q_1)$ and $G(\Q^{\dim q_1} q_2)$ are isomorphic.
\end{lemma}

\begin{proof}
    Let $E_1$, $E_2$ be edge sets of graphs $G(\Q^n, q_1)$ and $G(\Q^n, q_2)$ respectively.
    Let $T \in GL\left(n, \Q\right)$ such that  $q_1(\x) = q_2(T\x)$. We can see that
    \[ 
        (\x, \y) \in E_1 \Longleftrightarrow q_1(\x - \y) = 1 \Longleftrightarrow q_2(T\x - T\y) = 1 \Longleftrightarrow (T\x, T\y) \in E_2,
    \]
    Since $T$ is invertible we can see that $T$ is indeed an isomorphism of $G(\Q^n, q_1)$ and $G(\Q^n, q_2)$.
\end{proof}

\begin{definition}
    We will write $q_1 \embed q_2$ if $q_1$ can be built up to the equivalent of $q_2$, i.e. there exists $r \in \mcQ_{\dim q_1 - \dim q_2}$ such that quadratic form $h(\x, \y) := q_1(\x) + r(\y)$ holds $h \eqQ q_2$.
\end{definition}

\begin{lemma}\label{embed}
	If $q_1 \embed q_2$ then graph $G(\Q^{\dim q_1}, q_1) $ is a subgraph of $G(\Q^{\dim q_2}, q_2)$.
\end{lemma}

\begin{proof}
    By definition there exists $r \in \mcQ_{\dim q_1 - \dim q_2}$ such that quadratic form $h(\x, \y) := q_1(\x) + r(\y)$ holds $h \eqQ q_2$. Hence $G(\Q^{\dim q_2}, q_2) \simeq G(\Q^{\dim h}, h)$. If we constrain latter graph on the first $\dim q_1$ coordinates we will get $G(\Q^{\dim q_1}, q_1)$.
    
\end{proof}

Let us state a lemma that we are going to use later.
\begin{lemma}[Meyer, \cite{Serre1973}, chapter IV, theorem 8, corollary 2] \label{Meyer}
A quadratic form of rank $\geq 5$  represents 0 in $\Q$ if and only of it represents $0$ in $\R$ (i.e. there exists $\x \in \R^n$, $\x \ne \zero$ such that $q(\x) = 0$).
\end{lemma}

\begin{theorem}
For any $n \geq 4$ and $q \in \mcQ_n$ graph $G(\Q^n, q)$ is nonempty.
\end{theorem}

\begin{proof}
    $G(\Q^n, q)$ is nonempty iff there exists rational solution to $q(\x) = 1$. Consider following quadratic form

\[
    r(x_1, \ldots, x_n, x_{n+1}) = q(x_1, \ldots, x_n) - x_{n+1}^2
\]
This is indefinite quadratic form of rank $n + 1 \geq 5$. Hence by lemma \ref{Meyer} there exists nontrivial solution $(x_1, \ldots, x_{n+1})$ to the equation $r(x_1, \ldots, x_{n+1}) = 0$. It is easy to see that $x_{n+1} \ne 0$ (otherwise $q(x_1, \ldots, x_n) = 0$ hence $x_1 = \ldots = x_n = 0$ ) and $q\left( \frac{x_1}{x_{n+1}}, \ldots,  \frac{x_n}{x_{n+1}}\right) = 1$.
\end{proof}

The following statement shows that the latter theorem is exact, i.e. for $n < 4$ there exists $q \in \mcQ_n$ such that $G(\Q^n, q)$ is empty.

\begin{proposition}\label{23 and 233 
}
	Graphs $G(\Q^2, 2x^2 + 3y^2)$ and $G(\Q^3, 2x^2 + 3y^2 + 3z^2)$ are empty.
\end{proposition}

\begin{proof}
    Assume the contrary, i.e. there exists a rational solution to $2x^2 + 3y^2 + 3z^2 = 1$. Let $p, q, r, s$ be  minimal integers such that  \[ 2p^2 + 3q^2 + 3r^2 = s^2 \] 
    Then $2p^2 - s^2$ is divisible by 3 which means that $p$ and $s$ are divisible by 3. Then $2p^2 - s^2$ is divisible by 9 and $q^2 + r^2$ is divisible by 3. Then $q$ and $r$ are divisible by 3.
    This contradicts to the minimality of $p, q, r, s$ (we can divide them by 3).

    Since $G(\Q^2, 2x^2 + 3y^2)$ is a subgraph of $G(\Q^3, 2x^2 + 3y^2 + 3z^2)$, it is also empty.
\end{proof}

Next lemma gives the exact criteria of $G(\Q^n, q)$ nonemptiness for any $n$ and $q$.

\begin{lemma}\label{I_1 in q}
    Graph $G(\Q^n, q)$ is nonempty iff $I_1 \embed q$ or equivalently there exists $r \in \mcQ_{n - 1}$ such that 
    \[q(x_1, \ldots, x_n) \eqQ x_1^2 + r(x_2, \ldots, x_n)\]
\end{lemma}

\begin{proof}
If $q(x_1, \ldots, x_n) \eqQ x_1^2 + r(x_2, \ldots, x_n)$, then by lemma~\ref{isom} there holds $G(\Q^n, q) \simeq G(\Q^n, x_1^2 + r(x_2, \ldots, x_n))$ and we can find the edge $(1, 0, \ldots, 0)$ in the latter.

If graph $G(\Q^n, q)$ is nonempty, consider solution $v = (v_1, v_2, \ldots, v_n)$  of $q(\x) = 1$. We can take $T \in GL(n, \Q)$ such that $v$ is one of the coordinate axis (for example $x_1$). Now it is easy to see that $q(T\x)$ has the form $x_1^2 + r(x_2, \ldots, x_n)$.
\end{proof}

Chilakamarri \cite{chilakamarri1988unit} proved that graphs $G(\Q^2, I_2)$, $G(\Q^3, I_3)$ and $G(\Q^4, I_4)$ are disconnected. He also showed that for $n \geq 5$ graphs $G(\Q^n, I_n)$ are also connected.
The following theorem is the generalization of this statement.

\begin{theorem} \label{connect_for_geq_5}
    For any $n \geq 5$ and $q \in \mcQ_n$ graph $G(\Q^n, q)$ is connected. 
\end{theorem}

\begin{proof}

Consider $T \in GL(n, \Q)$ such that $q(T\x)$ is diagonal. We will redefine $q(\x)$ as $q(T\x)$, for by lemma \ref{isom} this will not change the considered graph.

Let $q(\x) = ax_1^2 + r(x_2, \ldots, x_n)$, $a > 0$, $r \in \mcQ_{n-1}$.

Consider any rational $0 < c < \frac{1}{\sqrt{a}}$ and quadratic form $q_1(\x) = r(x_2, \ldots, x_n) - (1 - ac^2) x_1^2$. By lemma~\ref{Meyer} there exists rational solution to $q_1(\x) = 0$. Let us divide this solution by $x_1^2$ which is clearly nonzero.
 
Now let $(x_2^0, \ldots, x_n^0)$ be the obtained solution to $r(x_2, \ldots, x_n) = 1 - ac^2$.

It is easy to see that points $(0, 0, \ldots, 0)$, $(c, x_2, \ldots, x_n)$ and $(2c, 0, \ldots, 0)$ form a path of length 2 in $G(\Q^n, q)$. Using that we can find path from $(0, 0, \ldots, 0)$ to $(2kc, 0, \ldots, 0)$ for any $k \in \N$. Since we can take arbitrary small $c$ and arbitrary large $k$, we can see that there exists path from $(0, 0, \ldots, 0)$ to $(x, 0, \ldots, 0)$ for any  $x \in \Q$.  

Applying this consequently for other coordinates we can see that graph $G(\Q^n, q)$ is connected.
\end{proof}

\section{Clique numbers of $G(\Q^n, q)$}\label{clique_chapter}

This chapter is devoted to finding the clique number $\omega(G(\Q^n, q))$ of $G(\Q^n, q)$, i.e. the maximum number of  pairwise connected vertices.

We will use notions and theorems from theory of rational quadratic forms that can be found in the appendix and in the book \cite{Serre1973}.

Let us denote
$$S_n(x_1, \ldots, x_n) = \sum_{1 \leq i \leq j \leq n} x_i x_j$$


\begin{theorem}\label{omega = n + 1}
$\omega\left(\Q^n, q\right) = n + 1$ iff $S_n \eqQ q$.
\end{theorem}

\begin{proof}
    If $q \eqQ S_n$ then $G(\Q^n, q) \simeq G(\Q^n, S_n)$. It is easy to check that in $G(\Q^n, S_n)$ points 
        \[ \e_0 = (0, 0, 0, \ldots, 0) \]
	\[ \e_1 = (1, 0, 0, \ldots, 0) \]
	\[ \e_2 = (0, 1, 0, \ldots, 0) \]
	\[ \ldots \]
	\[ \e_n = (0, 0, 0, \ldots, 1) \]
    
    form a clique of size $(n + 1)$.
    
    To the contrary, let $\omega\left(\Q^n, q\right) = n + 1$ and $v_0 = (0, 0, \ldots, 0), v_1, \ldots, v_n$ is $(n+1)$-clique in $G(\Q^n, q)$. Consider $T \in GL(n, \Q)$ such that $Tv_i = e_i$ for any $(1 \leq i \leq n)$. Let $\displaystyle r(\x) = q(T\x) = \sum_{i \leq j} r_{ij} x_i x_j$.
    Since $r(v_i - v_0) = 1$ we get $r_{ii} = 1$.
    Since $r(v_i - v_j) = 1$ we get   $r_{ii} + r_{jj} - r_{ij} = 1$ hence $r_{ij} = 1$. 
    So we obtain $r = S_n$.
\end{proof}

\begin{theorem}
\[
    \omega(G(\Q^n, q))  = \max_{k \colon S_k \embed q} k + 1
\]
\end{theorem}

\begin{proof}
    If $S_k \embed q$ then $G(\Q^k, S_k) \subseteq G(\Q^n, q)$ which means
    \[
    \omega(G(\Q^n, q)) \geq \omega(G(\Q^k, S_k)) = k + 1.
    \]

    Assume $\omega(G(\Q^n, q)) \geq k + 1$. Then consider affine subspace $V$ of dimension $k$ that contains $(k+1)$-clique. Then by theorem~\ref{omega = n + 1} in this subspace quadratic form $q$ is equivalent $S_k$, hence $S_k \embed q$.
\end{proof}

\begin{theorem}
    For any $q \in \mcQ_n$ there holds $\omega(G(\Q^n, q)) \geq n - 2$.
\end{theorem}

\begin{proof}
    By lemma \ref{diff>=3} from appendix $S_{n - 3} \embed q$. Using previous theorem we get $$\omega(G(\Q^n, q)) \geq \omega(G(\Q^{n - 3}, S_{n - 3})) = n - 2.$$
\end{proof}

\section{Graph isomorphisms}\label{isom_chapter}

As showed in lemma \ref{isom}, if $q_1 \eqQ q_2$ then $G(\Q^n, q_1) \simeq G(\Q^n, q_2)$.

\begin{definition}
    By $G_0(\Q^n, q)$ we denote a connected component of $G(\Q^n, q)$ that contains point $\zero$.
\end{definition}




\begin{conjecture}\label{isomorphism}
    If $G_0(\Q^n, q_1) \simeq G_0(\Q^n, q_2)$ then $q_1 \eqQ q_2$. Moreover any isomorphism between these graphs is a linear transformation $T$ such that $q_1(\x) = q_2(T\x)$.
\end{conjecture}

This conjecture indeed generalises rational Beckman--Quarles theorem since if we take $q_1 = q_2 = I_n$ it will state that any automorphism of $G_0(\Q^n, I_n)$ is a linear transformation that preserves quadratic form $I_n$. i.e. isometry.

\begin{lemma}\label{M eqQ q}
    Let $v_1, \ldots, v_n$ be linearly independent vectors in $\Q^n$. Then quadratic form  $m$ with matrix
    \[
    M = \Big\{(v_i, v_j)_q\Big\}_{i, j = 1}^{n} :=  \Big\{\frac{q(v_i) + q(v_j) - q(v_i - v_j)}{2}\Big\}_{i, j = 1}^{n}
    \]
    is rationally equivalent to $q$.
\end{lemma}

\begin{proof}
Consider Hilbert space $\Q^n$ with inner product $\langle \x, \y \rangle = (q(\x) + q(\y) - q(\x - \y)) / 2$.
Let $\e_1, \ldots, \e_n$ be a standard basis of $\Q^n$.

It is easy to see that $M$ is a Gram matrix of vectors $v_1, \ldots, v_n$ and the matrix $Q$ of quadratic form $q$ is a Gram matrix of vectors $e_1, \ldots, e_n$. Since $T$ transforms $e_1, \ldots, e_n$ into $v_1, \ldots, v_n$, we can see that $M = T^{T}QT$ which is equivalent to $m(\x) = q(T\x)$. Hence $m \eqQ q$.
\end{proof}

\begin{definition}
    Let $p_G(a, b, k)$ be a number of paths of length $k$ from $a$ to $b$ .
\end{definition}

\begin{lemma}
Let $G = G_0(\Q^n, q)$. Then $p_G(a, b, k) = 1 \Longleftrightarrow q(a - b) = k^2$.
\end{lemma}

\begin{proof}
Let $a = x_1, x_2, \ldots, x_{k+1} = b$ be vertices of the path of length $k$ from $a$ to $b$. For $1 \leq i \leq k$ consider $y_i = x_{i + 1} - x_i$ to be  the edges of this path. Then for any $i$ there holds $q(y_i) = 1$ and $y_1 + \ldots + y_k = b - a$. If not all edges are equal then there is $y_i \neq y_{i + 1}$. Then between $a$ and $b$ there will be another path that consists of edges
$$y_1, y_2, \ldots, y_{i - 1}, y_{i + 1}, y_{i}, y_{i + 2}, \ldots, y_n.$$ 

If $y_1 = y_2 = \ldots = y_n$, then $q(a - b) = q(k \cdot y_1) = k^2$. 
\end{proof}


\begin{conjecture}\label{rational simplex}
    For any $q \in \mcQ_n$ if graph $G_0(\Q^n, q)$ 
 is nonempty then there exists $(n+1)$-simplex with rational lengths in $(\Q^n, q)$.
\end{conjecture}

\begin{theorem} \label{conj2 -> conj1}
    Conjecture \ref{rational simplex} implies conjecture \ref{isomorphism}.
\end{theorem}

\begin{proof}
    Let  $q_1, q_2 \in \mcQ_n$, such that $G_1 = G_0(\Q^n, q_1)$ and $G_2 = G_0(\Q^n, q_2)$ are isomorphic. By conjecture~\ref{rational simplex} there exists a simplex with rational edges in $(\Q^n, q_1)$.
Because we can find a simplex with rational sides we can find a simplex with natural lengths, i.e. there are points $x_0, x_1, \ldots, x_n \in \Q^n$ such that for any $1 \leq i, j \leq n$ there holds $q_1(x_i - x_j) = \ell_{ij}^2$, where $\ell_{ij} \in \N$. 
    
    This means that there are points $x_1, x_2, \ldots, x_n \in \Q^n$, such that $p_{G_1}(x_i, x_j, \ell_{ij}) = 1$. Since $G_1 \simeq G_2$, there are points $y_0, y_1, \ldots, y_n \in \Q^n$, such that $p_{G_2}(y_i, y_j, \ell_{ij}) = 1$ and $q_2(y_i - y_j) = \ell_{i,j}^2$.

    Using lemma~\ref{M eqQ q} we get that forms $q_1$ and $q_2$ are both equivalent to a quadratic form with matrix $\Big\{\frac{\ell_{1i}^2 + \ell_{1j}^2 - \ell_{ij}^2}{2}\Big\}_{i, j = 1}^{n}$, therefore are equivalent one to another.
\end{proof}

\begin{theorem}
    Conjectures~\ref{rational simplex} and~\ref{isomorphism} are true provided by $n = 2$.
\end{theorem}

\begin{proof}
Consider $q \in \mcQ_2$ such that $G_0(\Q^2, q)$ is nonempty. Using theorem~\ref{I_1 in q} we can say that $q \eqQ x^2 + ny^2$ and we can check that there is a triangle with rational sides in $(\Q^2, x^2 + ny^2)$.

For this we need to consider triangle with coordinates $(0, 0), (n - 1, 2), (-n + 1, 2)$. Its sides are equal to $n + 1, n + 1$ and $2n - 2$.
\end{proof}

\begin{theorem}[Rational Beckman--Quarles theorem] 
Conjectures~\ref{rational simplex} and~\ref{isomorphism} hold $q_1 = q_2 = I_n$ $(n > 1)$.
\end{theorem}

\begin{proof}
Using theorem~\ref{conj2 -> conj1} we need to find rational simplex in $\Q^n$ with rational sides.

For even $n$ we can consider the following simplex:
\[
\begin{matrix}
(1, &  1, & 0, &  0, & \ldots & 0, & 0), \\
(1, & -1, & 0, &  0, & \ldots & 0, & 0), \\
(0, &  0, & 1, &  1, & \ldots & 0, & 0), \\
(0, &  0, & 1, & -1, & \ldots & 0, & 0), \\
\ldots &  \ldots & \ldots & \ldots & \ddots & \ldots & \ldots \\
(0, &  0, & 0, &  0, & \ldots & 1, & 1), \\
(0, &  0, & 0, &  0, & \ldots & 1, & -1), \\
(1, &  1 /15, & 0, &  0, & \ldots & 0, & 0) \\
\end{matrix}
\]

For odd $n$ we can consider the following simplex:
\[
\begin{matrix}
(1, &  1, & 0, &  0, & \ldots & 0, & 0, & 0), \\
(1, & -1, & 0, &  0, & \ldots & 0, & 0, & 0), \\
(0, &  0, & 1, &  1, & \ldots & 0, & 0, & 0), \\
(0, &  0, & 1, & -1, & \ldots & 0, & 0, & 0), \\
\ldots &  \ldots & \ldots & \ldots & \ddots & \ldots & \ldots & \ldots \\
(0, &  0, & 0, &  0, & \ldots & 1, & 1, & 0), \\
(0, &  0, & 0, &  0, & \ldots & 1, & -1, & 0), \\
(1, &  1/15, & 0, &  0, & \ldots & 0, & 0, & 0) \\
(0, &  0, & 0, &  0, & \ldots & 0, & 0, & 7/4) \\
\end{matrix}
\]

It is easy to see that all sides of these simplices are rational.

\end{proof}

Lemma \ref{embed} shows that if $q_1 \embed q_2$ then $G(\Q^{\dim q_1}, q_1) \subseteq G(\Q^{\dim q_2}, q_2)$. Does the contrary hold?

\begin{proposition}
Consider $q(x, y, z) = \frac{1}{3}(x^2 + y^2 + 2z^2)$. Then
    $G(\Q^2, I_2) \subset G(\Q^3, q_3)$, but $I_2 \notembed q$.
\end{proposition}

\begin{proof}
As shown in article \cite{Benda2000}, graph $G(\Q^2, I_2)$ is bipartite, let $G(\Q^2, I_2) = G_1 \sqcup G_2$ where $G_1, G_2$ --- are parts of this graph.
Consider following mapping $\phi: \Q^2 \to \Q^3$

    $$\phi(x, y) = \begin{cases}
    (x_1, x_2, 0), & \text{if } (x_1, x_2) \in G_1 \\
    (x_1, x_2, 1), & \text{if } (x_1, x_2) \in G_2
    \end{cases}
    $$

It is easy to see that $(\x, \y)$ is the edge of $G(\Q^2, I_2)$ then $(\phi(\x), \phi(\y))$ is the edge of $G(\Q^3, q_3)$.

Now using the  statement \ref{forms embed} we obtain that $I_2 \notembed q$.
    
\end{proof}

\appendix

\section{Rational quadratic forms}\label{appendix}

Let $\Nu = \{p$ : $p$ is prime$\} \cup \{\infty\}$. We denote by $\Q_{\infty} = \R$, and by $\Q_p$ the $p$-adic number field for a prime $p$.
We also denote $\Q_* = \Q \backslash \{0\}$ and  $\Q_*^2 = \{ x^2 | x \in \Q_* \}$. 

Let $A = \diag\left(a_1, \ldots, a_n\right),$  where $a_i \in \Q$ and $a_i> 0$. Then 

\begin{itemize}
	\item $\displaystyle D\left(A\right) = \prod_{i=1}^{n} a_i$ is a determinant of $A$
	\item For any $\nu \in \Nu$  we denote a \textit{Hasse--Minkowski invariant} $$E_{\nu}\left(A\right) = \prod_{i < j} \left(a_i, a_j\right)_{\nu} \in \{\pm 1\}$$
        where $\left(a, b\right)_{\nu}: \left(\Q_* / \Q_*^2\right)^2 \rightarrow \{\pm 1\}$ is the Hilbert symbol on $\Q_{\nu}$. 
\end{itemize}

Hilbert symbol $\left(a, b\right)_\nu$ equals 1 iff equation $z^2 = ax^2 + by^2$ has nontrivial solution $\left(x, y, z\right)$ over $\Q_\nu^3$.
This value is easily computable and has following properties \cite{Serre1973}:

\begin{multicols}{2}
\begin{itemize}
	\item $\left(a, b\right)_\nu = \left(b, a\right)_\nu$;
	\item  $\left(a, c^2\right)_\nu = 1$;
	
	\item $\left(a, -a\right)_\nu = \left(a, 1 - a\right)_\nu = 1$;

\columnbreak
	\item $\left(a, bc\right)_\nu = \left(a, b\right)_\nu \cdot \left(a, c\right)_\nu$;
	\item  $\left(a, ab\right)_\nu = \left(a, -b\right)_\nu$;

	\item $\displaystyle \prod_{\nu \in \Nu} \left(a, b\right)_\nu = 1$.
\end{itemize}

\end{multicols}


In \cite{Serre1973} it is shown that $D \mod \Q_{*}^2$ and $E_\nu$ for $\nu \in \Nu$ are the only invariants of rationally positive definite quadratic forms.

\begin{lemma}\label{forms embed}
	$A \embed B$ iff there exists matrix $X$ such that 
	\begin{itemize}
		\item $\dim X = \dim B - \dim A$; 
		\item $D\left(X\right) \equiv D\left(A\right)D\left(B\right) \mod \Q_*^2$;
		\item for every $\nu$ there holds $E_\nu\left(X\right) = E_\nu\left(B\right) \cdot E_\nu\left(A\right) \cdot \left(D\left(A\right), -D\left(B\right)\right)_\nu$.
	\end{itemize}
\end{lemma}

\begin{proof}
	Without loss of generality we assume that $A = \diag\left(a_1, \ldots, a_n\right)$ and $B = \diag\left(b_1, \ldots, b_m\right)$. $A \embed B$ iff there exists $X = \diag\left(x_{n+1}, \ldots, x_m\right)$ such that $C = \diag\left(a_1, \ldots a_n, x_{n+1}, \ldots, x_m\right) \eqQ B$.
	
	As stated above, for $C \eqQ B$ we need to satisfy the following two conditions:
	\begin{itemize}
		\item 	$D\left(C\right) \equiv D\left(B\right) \mod \Q_*^2$
		\item $E_\nu(C) = E_\nu(B)$ for every $\nu \in \Nu$.
	\end{itemize}
	
	$D\left(C\right) = D\left(A\right) D\left(X\right) \equiv D\left(B\right) \mod \Q_*^2$ so $D\left(X\right) \equiv D\left(A\right)D\left(B\right) \mod \Q_*^2$.
	
	Then for every $\nu \in \Nu$ we get 
	$$E_\nu\left(B\right) = E_\nu\left(C\right) = \prod_{i <j \leqslant n} \left(a_i, a_j\right)_\nu \cdot \prod_{i = 1}^{n} \prod_{j=n+1}^{m} \left(a_i, x_j\right)_\nu \cdot \prod_{n+1 \leqslant i < j } \left(x_i, x_j\right)_\nu = E_\nu\left(A\right) \cdot \left(D\left(A\right), D\left(X\right)\right)_\nu \cdot E_\nu\left(X\right)$$
	
	$$E_\nu\left(X\right) = E_\nu\left(A\right) \cdot E_{\nu}\left(B\right) \cdot \left(D\left(A\right), D\left(A\right)D\left(B\right)\right)_\nu  = E_\nu\left(A\right) \cdot E_{\nu}\left(B\right) \cdot \left(D\left(A\right), -D\left(B\right)\right)_\nu$$
\end{proof}

Now we need to state Proposition 7 in Chapter 4 from \cite{Serre1973}.

\begin{proposition}
	Let numbers $D$, $E_\nu$ and $\left(r, s\right)$	satisfy the following conditions:
	\begin{enumerate}
		\item $E_v = 1$ for almost all $\nu \in \Nu$ and $\prod_{\nu \in \Nu} E_\nu = 1$
		\item $E_\nu = 1$ if $n = 1$ or if $n = 2$ and image of $D_v$ of $D$ in $\Q_\nu / \Q_\nu^2$ equals $-1$.
		\item $r, s \geq 0$ and $r + s = n$
		\item $D_\infty = \left(-1\right)^s$, \quad  $E_\infty = \left(-1\right)^{s\left(s-1\right)/2}$ 
	\end{enumerate}
	Then there exists quadratic form of rank $n$ over $\Q$ with invariants $D, E_\nu$ and $\left(r, s\right)$, where $\left(r, s\right)$ is the signature of form $X$.
\end{proposition}

Combining these statements we obtain the following lemma.

\begin{lemma} \label{diff>=3}
Consider $q_1 \in \mcQ_n$ and $q_2 \in \mcQ_m$. If $m \geq n + 3$ then $q_1 \embed q_2$.
\end{lemma}

For every $k$ we define
$$
\lambda_k=
\begin{cases}
k + 1, \quad k \text{ is even; }\\
\left(k + 1\right)/2, \quad k \text{ is odd. }
\end{cases}
$$

We can see that $D\left(S_k\right) = \frac{k+1}{2^k} \equiv \lambda_k \mod \Q_*^2$

\begin{lemma}
	$$E_\nu\left(S_k\right) = \left(k + 1, \lambda_{k-1}\right)_\nu = 
	\begin{cases}
	\left(k + 1, -2\right)_\nu, \quad k \text{ is even; }\\
	\left(k + 1, -1\right)_\nu, \quad k \text{ is odd. }
\end{cases} 
	$$
\end{lemma}

\begin{proof}
We prove this statement by induction over $k$.
For $k = 1$ we get $E_\nu\left(S_k\right) = 1$ and $\left(2, -2\right)_\nu = 1$. Then for every $k$ we get

$$E_\nu\left(S_{k+1}\right) = E_\nu\left(S_k\right) \cdot \left(D\left(S_k\right), \frac{k+1}{2\left(k+2\right)}\right)_\nu = E_\nu\left(S_k\right) \cdot \left(\lambda_k, \frac{\left(k+2\right)\left(k+1\right)}{2}\right)_\nu = $$
$$ = \left(k + 1, \lambda_{k-1}\right)_\nu \cdot \left(\lambda_k, \frac{\left(k+1\right)\left(k+2\right)}{2}\right)_\nu = \left(k + 1, \lambda_{k}\right)_\nu \cdot \left(k + 1, \lambda_{k-1}\lambda_k\right)_\nu \cdot \left(\lambda_k, \frac{\left(k+1\right)\left(k+2\right)}{2}\right)_\nu = $$

$$= \left(\lambda_k, \frac{k+2}{2}\right)_\nu \cdot \left(k + 1, \frac{k\left(k+1\right)}{2}\right)_\nu = \left(\lambda_k, \frac{k+2}{2}\right)_\nu \cdot \left(k + 1, 2\right)_\nu = \left(\lambda_k, k+2\right)_\nu \cdot \left(\lambda_k,2\right)_\nu \cdot \left(k + 1, 2\right)_\nu$$

Thus we need to check for every $k$ that $\left(\lambda_k, 2\right)_\nu = \left(k + 1, 2\right)_\nu$.

If $k$ is even then this equality holds because $\lambda_k = k + 1$.

If $k$ is odd then $\left(\lambda_k, 2\right)_\nu = \left(\frac{k+1}{2}, 2\right)_\nu = \left(k + 1, 2\right)_\nu \cdot \left(2, 2\right)_\nu$. Now we only have to use the fact that $\left(2, 2\right)_\nu = 1$ holds for every $\nu \in \Nu$.

\end{proof}




\bibliographystyle{ieeetr}
\bibliography{main}

\begin{thebibliography}{10}

\bibitem{SokolovGit}
A.~Sokolov, ``Cliques in $\mathbb{Q}^n$.''
  \url{https://sokolovartemy.github.io/cliques_in_Qn.html}.

\bibitem{chilakamarri1988unit}
K.~B. Chilakamarri, ``Unit-distance graphs in rational $n$-spaces,'' {\em
  Discrete mathematics}, vol.~69, no.~3, pp.~213--218, 1988.

\bibitem{ElsholtzKlotz2005}
C.~Elsholtz and W.~Klotz, ``Maximal dimension of unit simplices,'' {\em
  Discrete Comput Geom}, vol.~34, pp.~167--177, 2005.

\bibitem{bau2021single}
S.~Bau, P.~Johnson, and M.~Noble, ``On single-distance graphs on the rational
  points in {E}uclidean spaces,'' {\em Canadian Mathematical Bulletin},
  vol.~64, no.~1, pp.~13--24, 2021.

\bibitem{Noble2018}
M.~Noble, ``Isosceles triangles in $\mathbb{Q}^3$,'' {\em Integers}, vol.~18,
  p.~A9, 2018.

\bibitem{noble2021embedding}
M.~Noble, ``Embedding {E}uclidean distance graphs in $\mathbb{R}^n$ and
  $\mathbb{Q}^n$,'' {\em arXiv preprint arXiv:2108.07713}, 2021.

\bibitem{beckman1953isometries}
F.~S. Beckman and D.~A. Quarles, ``On isometries of {E}uclidean spaces,'' {\em
  Proceedings of the American Mathematical Society}, vol.~4, no.~5,
  pp.~810--815, 1953.

\bibitem{zaks2001discrete}
J.~Zaks, ``A discrete form of the {B}eckman--{Q}uarles theorem for rational
  spaces,'' {\em Journal of Geometry}, vol.~72, no.~1, pp.~199--205, 2001.

\bibitem{zaks2003beckman}
J.~Zaks, ``The {B}eckman--{Q}uarles theorem for rational spaces,'' {\em
  Discrete mathematics}, vol.~265, no.~1-3, pp.~311--320, 2003.

\bibitem{zaksconnelly2003beckman}
J.~Zaks and R.~Connelly, ``The {B}eckman--{Q}uarles theorem for rational
  $d$-spaces, $d$ even and $d \geq 6$,'' in {\em Discrete Geometry} (A.~Bezdek,
  ed.), ch.~13, Boca Raton: CRC Press, 2003.

\bibitem{Serre1973}
J.-P. Serre, {\em A course in arithmetic}.
\newblock No.~7 in Graduate texts in mathematics, New York [u.a.]: Springer,
  1973.

\bibitem{Benda2000}
M.~Benda and M.~Perles, ``{Colorings of metric spaces},'' {\em Geombinatorics},
  vol.~9, no.~3, p.~113–126, 2000.

\end{thebibliography}


\end{document}